\documentclass[11pt]{amsart}  
\usepackage[T1]{fontenc}
\usepackage[dvips]{color}
\usepackage{amscd,amsfonts,amsmath,amssymb,
amsthm,amstext,latexsym,amstext}
\usepackage{enumerate,pifont,graphicx}
\usepackage[english]{babel}
\newcommand{\Z}{{\mathbb Z}}

\newcommand{\R}{{\mathbb R}}
\newcommand{\C}{{\mathbb C}}
\newcommand{\N}{{\mathbb N}}

\def \CC {\overline{\mathbb C\,}}

\def\boA{{\mathcal A}}

\def\boF{{\mathcal F}}
\def\boG{{\mathcal G}}
\def\boH{{\mathcal H}}

\def\boV{{\mathcal V}}
\def\boZ{{\mathcal Z}}

\def\cqfd{\hfill$\Box$}
\def\Res{{\,\rm Res}}

\def\Re{{\rm Re}}
\def\Im{{\rm Im}}

\def\disp{\displaystyle}

\newtheorem{theorem}{Theorem}[section]
\newtheorem{lemma}[theorem]{Lemma}
\newtheorem{proposition}[theorem]{Proposition}
\newtheorem{remark}[theorem]{Remark}

\newtheorem{definition}[theorem]{Definition}
\newtheorem{hypothesis}[theorem]{Hypothesis}
\newtheorem{example}[theorem]{Example}

\author{Filippo Morabito}
\author{Martin Traizet}

\thanks{This work started when first author 
was affiliated to Laboratoire de Mathématiques et Physique 
Théorique UMR CNRS 6083, Université François Rabelais,
Parc Grandmont, 37200 Tours, France.
First author was supported by an ANR "Minimales" grant.}
\address{
School of Mathematics, KIAS Korea Institute for Advanced Study,
207-43, Cheongnyangni 2-dong, Dongdaemun-gu, 
Seoul, 130-722, Korea}
\email{morabito@kias.re.kr}
\address{Laboratoire de Mathématiques et Physique 
Théorique UMR CNRS 6083, Université François Rabelais,
Parc Grandmont, 37200 Tours, France} 
\email{martin.traizet@lmpt.univ-tours.fr}
\title{Non-periodic Riemann examples with handles}
\begin{document}
\begin{abstract}
We show the existence of  $1$-parameter families 
of non-periodic, complete, embedded  minimal surfaces in
euclidean space with infinitely many parallel planar ends.
 In particular we are able to produce finite genus 
 examples and quasi-periodic examples of infinite genus.
  \end{abstract}
\maketitle
\section{Introduction}
\label{intro}

The goal of this paper is to construct families of
complete, embedded minimal surfaces in euclidean space,
with infinitely many planar ends.
The classical examples of such surfaces
have been discovered in 19th century
by B. Riemann and are called Riemann minimal examples.
They have genus zero and are periodic.
W. Meeks, J. Pérez, A. Ros \cite{MPR2}
proved that they are the only
properly embedded minimal surfaces of genus zero with
infinitely many ends.

\medskip
Periodic examples have been constructed by the second author
in \cite{T} by adding handles, in a periodic way, to Riemann
examples. Our goal in this paper is to follow the same strategy
without assuming any periodicity.
The first class of examples that we obtain have finite genus:
\begin{theorem}
\label{th1}
For each integer $g\geq 1$, there exists a 1-parameter family of
complete, properly embedded minimal surfaces which have genus $g$ and
infinitely many planar ends. These surfaces have two limit ends.
\end{theorem}
W. Meeks, J. Perez and A. Ros \cite{MPR} have proven 
that a properly embedded
minimal surface of finite genus and infinitely many ends, must have
planar ends and two limit ends.
(This later point means, in the finite genus case, that it is 
homeomorphic to a compact closed surface $M$ punctured in 
a countable set with precisely two limit points.)
Such surfaces have been constructed  by
L. Hauswirth and F. Pacard \cite{HP} using an analytic gluing
procedure.
Note however that the surfaces we construct are different: their 
examples
degenerate into Costa-Hoffman-Meeks surfaces whereas our 
examples degenerate into catenoids.

\medskip

We also obtain examples of infinite genus. In particular we can construct
quasi-periodic examples.
\begin{theorem}
\label{th2}
There exists complete, properly embedded minimal surfaces in euclidean
space which are quasi-periodic and non-periodic. These surfaces
have infinite genus, infinitely many planar ends and two 
limit ends.
\end{theorem}
Recall that a minimal surface $M$ is periodic if there exists a non-zero
translation $T$ such that $T(M)=M$.
We say that a minimal surface $M$ is quasi-periodic if
there exists a diverging sequence of translations $(T_n)_{n\in\N}$
such that the sequence
$(T_n(M))_{n\in\N}$ converges smoothly to $M$ on compact subsets of
$\R^3$ \cite{MPR3}.
(This notion of quasi-periodicity is weaker than the usual notion
of quasi-periodicity in cristallography.)
An example of quasi-periodic minimal surface in the flat manifold
$\R^2\times{\mathbb S}^1$ was constructed in \cite{MT}.
 
\medskip

To describe the surfaces we construct, we introduce some terminology.
Let $M$ be a complete, properly embedded minimal surface with
an infinite number of horizontal planar ends.
\begin{definition}
We say that $M$ is of type $(n_k)_{k\in\Z}$
if there exists an increasing sequence
$(h_k)_{k\in\Z}$ such that for each $k\in\Z$,
\begin{itemize}
\item the intersection of $M$ with the horizontal plane $x_3=h_k$
is the union of $n_k$ smooth Jordan curves,
\item the domain $h_k<x_3<h_{k+1}$ of $M$ has one planar end and
is homeomorphic to a planar domain.
\end{itemize}
\end{definition}
For example, Riemann examples are of type $(1)_{k\in\Z}$.
If $M$ is of type $(n_k)_{k\in\Z}$, then its genus is
equal to $\sum_{k\in\Z} (n_k-1)$, possibly infinite.
The following theorem is a particular case of our main result, to
be stated in the next section.
\begin{theorem}
\label{th3}
Let $(n_k)_{k\in\Z}$ be a sequence of positive integers. Assume that
\begin{itemize}
\item the sequence $(n_k)_{k\in\Z}$ is bounded,
\item for all $k\in\Z$, either $n_k=1$ or $n_{k+1}=1$.
\end{itemize}
Then there exists a family of complete properly embedded minimal surfaces
with infinitely many planar ends and of type $(n_k)_{k\in\Z}$.
Moreover, these surfaces are periodic (resp. quasi-periodic)
if and only if the sequence $(n_k)_{k\in\Z}$ is periodic
(resp. quasi-periodic).
\end{theorem}
We say that a sequence $(u_k)_{k\in\Z}$
is periodic if there exists a positive integer
$T$ such that $u_{k+T}=u_k$.
We say it is quasi-periodic if there 
exists a diverging sequence of integers $(T_n)_{n\in\N}$ such that
for all $k\in\Z$, $\disp\lim_{n\to\infty}u_{k+T_n}=u_k$.
(In the case of a sequence of integer numbers, this means that
for all $k\in\Z$, there exists an integer $N$ such that
for all $n\geq N$, $u_{k+T_n}=u_k$. In other words, any
finite portion of the sequence is repeated infinitely many times.)
It is clear that theorems \ref{th1} and \ref{th2} follow from this
theorem by choosing appropriate sequences $(n_k)_{k\in\Z}$.

\medskip

Heuristically, our surfaces are constructed by taking an infinite
stack of horizontal planes $(P_k)_{k\in\Z}$, ordered by their
height, and gluing $n_k$ catenoidal necks between the planes $P_k$
and $P_{k+1}$ for each $k\in\Z$.
It is clear that in this way, we obtain a surface of type
$(n_k)_{k\in\Z}$.

\medskip

The construction follows the lines of the one in the 
periodic case in \cite{T}.
In the periodic case, we worked in the quotient by the period, so
we only had to glue a finite number of catenoids and we could use
the classical theory of compact Riemann surfaces.
In the non-periodic case that we consider in this paper, we have to glue
infinitely many catenoids at the same time, and the underlying Riemann
surface is not compact.
This paper is part of a project where we develop tools to glue
infinitely many minimal surfaces together.
Technically, it relies on the theory of {\em opening
infinitely many nodes} developped by the second author in \cite{TT}.
\section{Configurations and forces, main result}
\label{section-configurations}
Let $(n_k)_{k\in\Z}$ be a sequence of positive integers.
A configuration of type $(n_k)_{k\in\Z}$ is a sequence
of complex numbers $(p_{k,i})_{1\leq i\leq n_k,k\in\Z}$.
The points $p_{k,i}$, $1\leq i\leq n_k$,
represent the position of the catenoidal necks
that we will create between the planes $P_k$ and $P_{k+1}$.
These points must satisfy a balancing condition which we
express in term of forces. Let $c_k=\frac{1}{n_k}$.
\begin{definition}
The force $F_{k,i}$ exerted on $p_{k,i}$ 
by the other points of the configuration
is defined as
\begin{equation}
\label{forces}
F_{k,i}=2\sum_{j=1\atop j \neq i}^{n_k}
 \frac{c_k^2}{p_{k,i}-p_{k,j}}-
 \sum_{j=1}^{n_{k+1}} \frac{c_k c_{k+1}}{p_{k,i}-p_{k+1,j}}
 - \sum_{j=1}^{n_{k-1}} \frac{c_k c_{k-1}}{p_{k,i}-p_{k-1,j}}.
\end{equation}
A configuration $\{p_{k,i}\}_{k \in \Z,i\in\{1,\ldots,n_k\}}$
 is said to be balanced if all forces $F_{k,i}$
vanish.
\end{definition} 
For the forces to be defined, we need that for each $k\in\Z$, the points
$p_{k,i}$, $1\leq i\leq n_k$ and $p_{k\pm 1,i}$, $1\leq i\leq n_{k\pm 1}$
are distinct, which we assume from now on.
We will see later the existing relationship between
the balancing condition and the period problem that we have to solve
to construct our family of minimal surfaces.

\begin{example}
\label{example1}
\em
Fix some non-zero complex number $a$.
The configuration
given by $n_k=1$ and $p_{k,1}=ka$ for all $k\in\Z$ is balanced.
This configuration yields the family of Riemann examples.
\end{example}
Moreover, as our construction is based on the implicit function theorem,
we need the differential of the force map to be invertible in some
sense.
In the case of example \ref{example1},
the differential of $F_{k,1}$ is given by
$$dF_{k,1}=\frac{1}{a^2}(2dp_{k,1}-dp_{k-1,1}-dp_{k+1,1}).$$
The operator $(u_k)_{k\in\Z}\mapsto (2 u_k-u_{k-1}-u_{k+1})_{k\in\Z}$
is neither injective nor surjective from $\ell^{\infty}(\Z)$
to $\ell^{\infty}(\Z)$, where $\ell^{\infty}(\Z)$ is the space
of bounded sequences $(u_k)_{k\in\Z}$ with the sup norm.
This motivates the following change of variables.
\begin{equation}
\label{change}
\left\{
\begin{array}{l}
u_{k,i}=p_{k,i}-p_{k,1},\quad 1\leq i\leq n_k,\quad k\in\Z\\
\ell_k=p_{k,1}-p_{k-1,1},\quad k\in\Z.\\
\end{array}
\right.
\end{equation}
By definition we have $u_{k,1}=0$.
We denote by $\bf U $ the sequence
\begin{equation}
\label{u}
\ldots,\ell_k,u_{k,2},\ldots,u_{k,n_k},
\ell_{k+1},u_{k+1,2},\ldots, u_{k+1,n_{k+1}},\ell_{k+2},\ldots
\end{equation}
The parameter $\bf U$ determines the configuration up to
a translation, which is irrelevant since the forces are invariant by
translation of the configuration.

The expression of the forces in terms of the new variables
is
$$F_{k,i}=2\sum_{j=1\atop j \neq i}^{n_k}
 \frac{c_k^2}{u_{k,i}-u_{k,j}}-
 \sum_{j=1}^{n_{k+1}} \frac {c_k c_{k+1}}
 {u_{k,i}-\ell_{k+1}-u_{k+1,j}}
 - \sum_{j=1}^{n_{k-1}} \frac{c_k  c_{k-1}}
 {u_{k,i}+\ell_{k}-u_{k-1,j}}.
$$
We define 
$$
G_{k}=
\sum_{i=1}^{n_{k}}
\sum_{j=1}^{n_{k-1}} 
\frac{c_k c_{k-1}}{p_{k,i}-p_{k-1,j}}=
\sum_{i=1}^{n_{k}}
\sum_{j=1}^{n_{k-1}} 
\frac{c_k c_{k-1}}{u_{k,i}+\ell_{k}-u_{k-1,j}}.
$$
An elementary computation gives
$$\sum_{i=1}^{n_k} F_{k,i}=G_{k+1}-G_k.$$
Therefore, the configuration is balanced if and only if
for all $k\in\Z$, we have
$F_{k,i}=0$ for $2\leq i\leq n_k$ and
$G_k=G_0$.
We denote by $\tilde{\bf F}$ the sequence 
\begin{equation}
\label{Gk}
\ldots,G_k,F_{k,2},\ldots,F_{k,n_k},
G_{k+1},F_{k+1,2},\ldots, F_{k+1,n_{k+1}},G_{k+2},\ldots 
\end{equation}
The configuration is balanced if and only if $\tilde{\bf F}=0$.
\begin{definition}
\label{def.iso}
We say that the configuration is non-degenerate if the differential
of $\tilde{\bf F}$ with respect to ${\bf U}$,
from $\ell^{\infty}(\Z)$ to itself, exists and
is an isomorphism.
\end{definition}
Let us return to the configuration of example \ref{example1} and
see that it is non-degenerate.
The configuration is given by $\ell_k=a$.
We have ${\bf U}=(\ell_k)_{k\in\Z}$ and
$\tilde{\bf F}=(G_k)_{k\in\Z}=(\frac{1}{\ell_k})_{k\in\Z}$. 
The map ${\bf U}\mapsto\tilde{\bf F}$ is differentiable with
differential equal to $\frac{-1}{a^2}id$, so the configuration is
non-degenerate.

\medskip

We are ready to state the main result of this paper.
\begin{theorem}
\label{th4}
Consider a balanced, non-degenerate configuration ${\bf U}$ of type
$(n_k)_{k\in\Z}$.
Further assume that
\begin{enumerate}
\item the sequence $(n_k)_{k\in\Z}$ is bounded,
\item the sequence ${\bf U}$ takes a finite number of values,
\item for all $k\in\Z$, $\disp\frac{1}{n_k}\sum_{i=1}^{n_k}p_{k,i}
\neq \disp\frac{1}{n_{k-1}}\sum_{i=1}^{n_{k-1}}p_{k-1,i}$.
\end{enumerate}
Then there exists a 1-parameter family $(M_t)_{0<t<\varepsilon}$
of complete, properly embedded
minimal surfaces with infinitely many planar ends, of type $(n_k)_{k\in\Z}$.
Furthermore, each surface $M_t$ is periodic (resp. quasi-periodic)
if and only if the configuration is periodic (resp. quasi-periodic).
\end{theorem}
We say that the configuration is periodic if there exists a positive integer
$T$ such that for all $k\in\Z$, $n_{k+T}=n_k$, $\ell_{k+T}=\ell_k$
and $u_{k+T,i}=u_{k,i}$ for all $2\leq i\leq n_k$.
We say the configuration is quasi-periodic if there exists a
diverging sequence of integers $(T_n)_{n\in\N}$ such that
for all $k\in\Z$, there exists an integer $N$ such that for
$n\geq N$, $n_{k+T_n}=n_k$, $\ell_{k+T_n}=\ell_k$
and $u_{k+T_n,i}=u_{k,i}$ for $2\leq i\leq n_k$.
(In other words, any finite part of the configuration is
repeated infinitely many times.)

\begin{remark}\em
\label{remark1}
Let us discuss the various hypotheses of the theorem.
\begin{itemize}
\item
We formulated the definition of non-degeneracy 
by using the $\ell^{\infty}$ norm.
It is certainly the case that for many configurations, this is not
the right norm to use: they are non-degenerate for a suitable choice
of the norms on both the domain and the target space of the differential
of the force map. We chose the $\ell^{\infty}$ norm because this is
the most natural one and there are already plenty of configurations
which are non-degenerate for this norm.
\item
It would certainly be interesting to allow for unbounded sequences
$(n_k)_{k\in\Z}$, but such configurations cannot be non-degenerate
for the $\ell^{\infty}$ norm. Maybe a result is possible using a weighted
$\ell^{\infty}$ norm, with the weight depending in some way on $n_k$.
\item Hypothesis 2 is not required in any fundamental way, but makes
the proof of the theorem so much easier. It ensures the finiteness
hypothesis \ref{hypothesis-finitude}, see section \ref{section-notations}.
For the examples of configuration that we will consider, hypothesis 2
is a consequence of hypothesis 1, see remark \ref{remark2}.
\item We do not know of any example of balanced configuration for which
hypothesis 3 fails. It ensures that the Gauss map has multiplicity 2
at the ends and makes the proof slightly simpler at one point (see section \ref{section-gauss-map}).
Theorem \ref{th4} definitely holds without this hypothesis.
\end{itemize}
\end{remark}
\section{Examples of balanced non-degenerate configurations}
\label{section-exemples}
In this section, we obtain balanced, non-degenerate
configurations by concatenation of finite configurations.
\subsection{Concatenation of finite configurations}
\medskip
Let $h$ be a positive integer and $(n_0,\cdots,n_h)$ be
a finite sequence of positive integers, such that
$n_0=n_h=1$.
A finite configuration of type $(n_0,\cdots,n_h)$ is a collection $C$
of complex numbers $(p_{k,i})_{0\leq k\leq h,1\leq i\leq n_k}$.
We call the points $p_{0,1}$ and $p_{h,1}$ respectively the
first and last point of the configuration.
We call $h$ the height of the configuration.
We call width of the configuration the quantity 
$\max\{n_0,\cdots,n_h\}.$
As in section \ref{section-configurations} we make the change of variables
$$\left\{
\begin{array}{l}
u_{k,i}=p_{k,i}-p_{k,1},\quad 1\leq i\leq n_k,\quad 0\leq k\leq h\\
\ell_k=p_{k,1}-p_{k-1,1},\quad 1\leq k\leq h.\\
\end{array}
\right.$$
The forces $F_{k,i}$ for $0\leq k\leq h$ and $1\leq i\leq n_k$
are defined as in section \ref{section-configurations}, with the convention that $n_{-1}
=n_{h+1}=0$. The quantities $G_k$ for $1\leq k\leq h$ are
defined in the same way.
\begin{definition}
\label{finite.block}
A finite configuration $C$ of type $(n_0,\cdots,n_h)$ is said to be
\begin{itemize}
\item balanced if 
$F_{k,i}=0$ for $1\leq k\leq h-1$ and $1\leq i\leq n_k$.
Note that we do not require that the forces $F_{0,1}$ and $F_{h,1}$ vanish.
We will denote by $F_C$ the value of $F_{0,1}$
and we will call it the residual
force of the configuration.
\item non-degenerate if the differential
of the map which associates to the vector
\begin{equation}
\label{input}
(\ell_1,u_{1,2},\ldots,u_{1,n_1},
\ell_{2},u_{2,2},\ldots, u_{2,n_{2}},\ell_{3},\ldots,\ell_{h})
\end{equation}
the vector
\begin{equation}
\label{output}
(G_1,F_{1,2},\ldots,F_{1,n_1},
G_{2},F_{2,2},\ldots, F_{2,n_{2}},G_{3}, \ldots,G_h)
\end{equation}
is an isomorphism.
\end{itemize}
\end{definition}
\begin{proposition}
\label{proposition-residual}
If $C$ is a finite, balanced configuration of height $h$, it holds
$$F_{h,1}=-F_{0,1}=-F_C,$$
$$(p_{h,1}-p_{0,1})F_C=\sum_{k=1}^h\frac{1}{n_k}.$$
In particular, the residual force never vanishes.
\end{proposition}
\begin{proof}
The proposition comes from the following two formulae, which hold
for any configuration of height $h$ (not necessarily balanced).
$$\sum_{k=0}^{h}\sum_{i=1}^{n_k} F_{k,i}=0$$
$$\sum_{k=0}^h\sum_{i=1}^{n_k} p_{k,i}F_{k,i}
=\sum_{k=0}^h n_k (n_k-1) c_k^2 - \sum_{k=0}^{h-1} n_k n_{k+1}
c_k c_{k+1}
=1-\sum_{k=0}^{h}\frac{1}{n_k}.$$
We omit the proof of these formulae.
\end{proof}
\begin{remark} \em
\label{remark2}
Let us fix the height $h$ and the type of the configuration,
namely the sequence $(n_0,\cdots,n_h)$.
We observe that the balancing condition can
be written as a finite system of polynomial equations in the
unknowns $\ell_k$ for $1\leq k\leq h$ and $u_{k,i}$ for
$1\leq k\leq h-1$ and $2\leq i\leq n_k$. 
Each polynomial equation defines an algebraic variety in $\C^n$,
where $n$ is the number of unknowns.
The set of balanced configurations of type
$(n_0,\cdots,n_h)$ is the intersection of these varieties.
There might be components of non-zero dimension, but a
non-degenerate configuration cannot belong to such a component.
Basic results of algebraic geometry ensure that there is only
a finite number of dimension zero components (which are points),
so there is at most
a finite number of balanced, non-degenerate configurations of type
$(n_0,\cdots,n_h)$.
\end{remark}
\begin{definition}
We will say that two configurations $C_1,C_2$ of finite 
height are compatible if their residual forces
are equal, that is $F_{C_1}=F_{C_2}.$
\end{definition}
Given a sequence of finite configurations $(C_m)_{m\in\Z}$,
we define their concatenation $C$ as follows.
We denote by $h_{m}$ the height of $C_{m}$.
We write $n_k^{(m)}$, $p_{k,i}^{(m)}$, $\ell_k^{(m)}$,
$u_{k,i}^{(m)}$ and $F_{k,i}^{(m)}$
for the quantities associated to the
configuration $C_m$.
Let $(\varphi_m)_{m\in\Z}$ be the sequence defined
inductively by
$\varphi_0=0$ and $\varphi_{m+1}=\varphi_m+h_m$ for $m\in\Z$.
We define the sequence $(n_k)_{k\in\Z}$ by
$$n_{\varphi_m+k}=n_k^{(m)}
\quad\mbox{for $m\in\Z$, $0\leq k\leq h_m$}$$
which makes sense because $n_{h_m}^{(m)}=1=n_{0}^{(m+1)}$.
We define the configuration $C$, of type $(n_k)_{k\in\Z}$, by
$$\left\{\begin{array}{l}
\ell_{\varphi_m+k}=\ell_k^{(m)} \quad\mbox{for $m\in\Z$, $1\leq k\leq h_m$},\\
u_{\varphi_m+k,i}=u_{k,i}^{(m)}\quad\mbox{for $m\in\Z$,
$1\leq k\leq h_m-1$, $2\leq i\leq n_k^{(m)}$}. 
\end{array}\right.$$
This amounts to translate the configurations so that for each $m\in\Z$,
the last point of $C_m$ coincides with the first point of $C_{m+1}$,
and identify these two points.
The following result is a generalization of 
Proposition 2.3 in \cite{T}.
\begin{proposition}
\label{combination}
Let $(C_m)_{m \in \Z}$ be a sequence of finite configurations.
Let $C$ be the configuration obtained by concatenation of these 
configurations, as explained above. Then:
\begin{itemize}
\item if all configurations $C_m$ are balanced and compatible, then
the configuration $C$ is balanced,
\item if moreover, all configurations $C_m$ are non-degenerate and
have height and width bounded by some number independant of $m$,
then the configuration $C$ is non-degenerate and
the sequence ${\bf U}$ defined by equation
\eqref{u} takes a finite number of values.
\end{itemize}
\end{proposition}
\begin{proof}

\medskip

Let us write, for $m\in\Z$,
$${\bf U}^{(m)}=(\ell_1^{(m)},u_{1,2}^{(m)},\ldots,u_{1,n_1}^{(m)},
\ell_{2}^{(m)},u_{2,2}^{(m)},\ldots, u_{2,n_{2}}^{(m)},\ell_{3}^{(m)},\ldots,\ell_{h}^{(m)}),
$$
$$\tilde{\bf F}^{(m)}=
(G_1^{(m)},F_{1,2}^{(m)},\ldots,F_{1,n_1}^{(m)},
G_{2}^{(m)},F_{2,2}^{(m)},\ldots, F_{2,n_{2}}^{(m)},G_{3}^{(m)}, \ldots,G_h^{(m)})$$
for the parameters and forces corresponding to the configuration
$C_m$.
Then we have, for the configuration $C$,
${\bf U}=({\bf U}^{(m)})_{m\in\Z}$ by definition of the concatenation and
$\tilde{\bf F}=(\tilde{\bf F}^{(m)})_{m\in\Z}$ by inspection.
If all configurations $C_m$ are balanced and compatible, then
all $G_k$ are equal and the configuration $C$ is balanced.

\medskip

Let us assume that the configurations $C_m$ are non-degenerate and
have bounded height and width.
Then there is only a finite number of possibilities
for the types of the configurations. 
By remark \ref{remark2},
there is only a finite number of configurations $C_m$ for $m\in\Z$
(some of them are repeated infinitely many times in the sequence
$(C_m)_{m\in\Z}$).
Hence the sequence ${\bf U}$ takes only a finite number of values.

\medskip

To prove that $\tilde{\bf F}$ is differentiable with respect to
${\bf U}$, we observe that for each $m\in\Z$,
$\tilde{\bf F}^{(m)}$ only depends on ${\bf U}^{(m)}$.
Since there is only a finite number of distinct configurations,
the differential of $\tilde{\bf F}^{(m)}$ with respect to
${\bf U}^{(m)}$ has norm bounded by some number independent of $m$,
and the same is true for the second order differential. 
This implies
that $\tilde{\bf F}$ is differentiable with respect to
${\bf U}$ from $\ell^{\infty}$ to itself,
with differential given by
$$d\tilde{\bf F}({\bf U})({\bf X})
=\left(d\tilde{\bf F}^{(m)}({\bf U}^{(m)})({\bf X}^{(m)})\right)_{m\in\Z}.$$
(In other words, the differential has a block diagonal structure.)
Again, since there is only a finite number of distinct configurations,
the norms of the inverses of $d\tilde{\bf F}^{(m)}({\bf U}^{(m)})$ are bounded by
some number independent of $m$, so
the operator
$${\bf X}\mapsto\left(d\tilde{\bf F}^{(m)}({\bf U}^{(m)})^{-1}({\bf X}^{(m)})\right)_{m\in\Z}$$
is bounded from $\ell^{\infty}$ to itself, so
$d\tilde{\bf F}({\bf U})$ is invertible and the configuration
$C$ is non-degenerate.
\end{proof}
\subsection{Examples of finite configurations}

\begin{example}
\label{example2}
\em
A trivial example:
fix some non-zero complex number $a$ and some positive
integer number $h$.
The configuration of height $h$ and
type $(1,1,\cdots,1)$ defined by
$p_{k,1}=ka$ for $0\leq k\leq h$
is balanced, non-degenerate and has residual force
equal to $\frac{1}{a}$.
\end{example}
\begin{example}
\label{example3}
\em
A nice family of configurations of height $2$ and
type $(1,n,1)$, where $n\in\N^*$,
which comes from \cite{T}. It is given by
$$p_{0,1}=0,\quad p_{2,1}=2{\rm i},$$
$$p_{1,j}={\rm i}+\cot \frac{j\pi}{n+1},\quad \mbox{ for $1\leq j\leq n$}.$$
\end{example}
\begin{proposition}
This configuration is balanced, non-degenerate
and has residual force $\disp\frac{n+1}{2n{\rm i}}$.
\end{proposition}
\begin{proof}
This configuration is proven to be balanced in \cite{T}, proposition
2.1. Let us prove that it is non-degenerate. It is enough
to prove that the differential is injective.
Let $X$ be an element in its kernel. Consider a path
${\bf U}(t)=(\ell_1(t),u_{1,2}(t),\cdots,u_{1,n}(t),\ell_2(t))$
in the parameter space, such that ${\bf U}(0)$ is the given
configuration, and ${\bf U}'(0)=X$. Then $G_1'(0)$, $G_2'(0)$
and $F_{1,i}'(0)$, $2\leq i\leq n$ all vanish because $X$
is in the kernel.
The points of the configuration at time $t$ are given by
$p_{0,1}(t)=0$,
$p_{1,i}(t)=\ell_1(t)+u_{1,i}(t)$ for $1\leq i\leq n$
and $p_{2,1}(t)=\ell_1(t)+\ell_2(t)$ (with $u_{1,1}(t)=0$).
We extend this configuration into a periodic configuration, denoted
$\widetilde{p}_{k,i}(t)$ of
period $T(t)=\ell_1(t)+\ell_2(t)$ by writing
$\widetilde{p}_{2k,1}(t)=kT(t)$ and
$\widetilde{p}_{2k+1,i}(t)=p_{1,i}(t)+kT(t)$ for $k\in\Z$.
Let us write $\widetilde{F}_{k,i}(t)$ for the forces of this
configuration, we have 
$$\widetilde{F}_{2k,1}(t)=F_{0,1}(t)+F_{2,1}(t)=G_1(t)-G_2(t),$$
$$\widetilde{F}_{2k+1,i}(t)=F_{2k+1,i}(t).$$
Hence, the derivatives of these forces at time $0$ all vanish.
Next, we scale this configuration by $1/T(t)$ so that its period is
constant by writing $\widehat{p}_{k,i}(t)=\frac{\widetilde{p}_{k,i}(t)}{T(t)}$.
If we write $\widehat{F}_{k,i}(t)$ for the forces of this configuration,
we have $\widehat{F}_{k,i}'(0)=0$. Since this periodic configuration is non-degenerate by proposition 2.1 in \cite{T},
in the sense given just after theorem 1.4 in the same paper,
we have $\widehat{p}_{k,i}'(0)=0$ for all $(k,i)$.
From this we get, for $0\leq k\leq 2$
$$p_{k,i}'(0)=T'(0)\widehat{p}_{k,i}(0)=\lambda p_{k,i}(0)
\quad\mbox{ with $\lambda=\frac{T'(0)}{T(0)}$.}$$
Then we write
$$G_2(t)=\frac{1}{n}\sum_{j=1}^n\frac{1}{p_{2,1}(t)-p_{1,j}(t)},$$
$$G_2'(0)=\frac{1}{n}\sum_{j=1}^n \frac{-1}{(p_{2,1}-p_{1,j})^2}
(\lambda p_{2,1}-\lambda p_{1,j})=-\lambda G_2(0).$$
Since $G'_2(0)=0$ and
$G_2(0)=F_C\neq 0$, this gives $\lambda=0$, so $p_{k,i}'(0)=0$
and $X=0$. This proves that the configuration is non-degenerate.
We compute the residual force using proposition \ref{proposition-residual}
\end{proof}
\begin{remark}\em We can scale these configurations by $\frac{2n}{n+1}$
so that they are compatible. Then by proposition \ref{combination},
we can concatenate them to obtain balanced
configuration whose type is any bounded sequence $(n_k)_{k\in\Z}$ such that
$n_k=1$ for even $k$.
This proves theorem \ref{th3}.
\end{remark}
\begin{example}
\label{example4}
An example of height 3 and type $(1,2,2,1)$ given by
$$p_{0,1}=0,\quad
p_{1,1}=\frac{-\sqrt2}{2}+{\rm i},\quad
p_{1,2}=\frac{\sqrt2}{2}+{\rm i},$$
$$
p_{2,1}=\frac{-\sqrt2}{2}+2{\rm i},\quad
p_{2,2}=\frac{\sqrt2}{2}+2{\rm i},\quad
p_{3,1}=3{\rm i}.$$
\end{example}
\begin{proposition}
This configuration is balanced, non-degenerate and
has residual force
$\frac{2}{3{\rm i}}.$ 
\end{proposition}
The proof is purely computational, we omit the details.
(The determinant of the matrix associated to the 
differential, computed with
Mathematica, equals $4/243$).

\section{Proof of the main theorem}
In this section we prove theorem \ref{th4}.
\subsection{Notations and parameters}
\label{section-notations}
There are six parameters in the construction, denoted by
$t$, $a$, $b$, $\alpha$, $\beta$ and $\gamma$.
The parameter $t$ is a positive real number.
All other five parameters are sequences
of complex numbers
of the form $u=(u_{k,i})_{k\in\Z,1\leq i\leq n_k}$.
We use the notation $u_k=(u_{k,i})_{1\leq i\leq n_k}\in\C^{n_k}$.

\medskip

The $\ell^{\infty}$ norm of the sequence
$u$ is defined as usual as
$||u||_{\infty}=\sup \{|u_{k,i}| \,:\, k\in\Z,\,1\leq i\leq n_k\}$.
Each parameter varies in a neighborhood for the $\ell^{\infty}$
norm of a central value. The central value is denoted by
an upperscript $0$, so the central value of the parameters
are denoted by
$a^0$, $b^0$, $\alpha^0$, $\beta^0$ and
$\gamma^0$.
Most of our statements only hold in a small neighborhood of the
central values.
\medskip

\medskip

The central value of the parameters will be given later.
An important point is that the
following hypothesis holds.
We say that a sequence $u=(u_{k,i})_{k\in\Z,1\leq i\leq n_k}$
is finitely valued if
the set $\{u_{k,i}\,:\, k\in\Z,\,1\leq i\leq n_k\}$
is finite.
\begin{hypothesis}[Finiteness hypothesis]
\label{hypothesis-finitude}
The central value of each parameter is finitely valued.
\end{hypothesis}
This will be useful to make various statements uniform with
respect to $k\in\Z$.
By a uniform constant, we mean some number which only depends
on the central value of the parameters.
We use the notation $D(a,r)$ for the disk of center $a$ and
radius $r$ in $\C$.
\subsection{Opening nodes}
Consider infinitely many copies of the Riemann sphere,
labelled $(\CC_k)_{k\in\Z}$.
We denote by $\infty_k$ the point $\infty$ in $\CC_k$, and 
$\C_k=\CC_k\setminus\{\infty_k\}$.
For each $k\in\Z$ and each $1\leq i\leq n_k$, select a point
$a_{k,i}\in\C_k$ and a point $b_{k,i}\in\C_{k+1}$. Identify these two
points to create a node. This defines a Riemann surface with nodes which
we call $\Sigma_0$.
The parameters involved in
this construction are
the sequences $a=(a_{k,i})_{k\in\Z,1\leq i\leq n_k}$ and
$b=(b_{k,i})_{k\in\Z,1\leq i\leq n_k}$.
The central value of these parameters is given, in term of the 
given configuration, by
\def\conj{{\rm conj}}
$$a_{k,i}^0=(-1)^k\conj^k(\ell_k^0+u_{k,i}^0)$$
$$b_{k,i}^0=(-1)^{k+1}\conj^{k+1}(u_{k,i}^0)$$
where $\conj(z)=\overline{z}$ denotes conjugation, so $\conj^k(z)$ is equal to $z$
if $k$ is even and $\overline{z}$ if $k$ is odd.
Observe that $a^0$ and $b^0$
are finitely valued
by hypothesis 2 of theorem \ref{th4}.

\medskip
For any $k\in \Z$, the points
$a_{k,i}^0$ for $1\leq i\leq n_k$ and $b_{k-1,i}^0$ 
for $1\leq i\leq n_{k-1}$ are distinct in $\C_k$. 
That follows from the identity 
$$l^0_k+u^0_{k,i}-u^0_{k-1,j}=p^0_{k,i}-p^0_{k-1,j}$$
and the fact that second member does not vanish
by construction.

\medskip
 
Let $\varepsilon_k>0$ be the smallest value of the
 distance between these points in $\C_k$.
Because of the finiteness hypothesis,
$\{\varepsilon_k\,:\,k\in\Z\}$
is finite so we can take
$\varepsilon=\min\{\varepsilon_k\,:\,k\in\Z\}>0$.
If $||a-a^0||_{\infty}<\frac{\varepsilon}{4}$ and
$||b-b^0||_{\infty}<\frac{\varepsilon}{4}$,
then for any $k\in\Z$, the points
$a_{k,i}$ for $1\leq i\leq n_k$ and $b_{k-1,i}$ for $1\leq i\leq n_{k-1}$ are at distance greater than $\frac{\varepsilon}{2}$
from each other, so they remain distinct.
We will be using this kind of argument very often. We will not
enter in details anymore and simply refer to the finiteness
hypothesis.

\medskip
 
For each $k\in\Z$ we consider a function $g_k$ defined 
on $\CC_k$ by
$$g_k(z)=\sum_{i=1}^{n_{k-1}}\frac{\beta_{k-1,i}}{z-b_{k-1,i}}
-\sum_{i=1}^{n_k}\frac{\alpha_{k,i}}{z- a_{k,i}}.$$
The new parameters are the sequences
$$\alpha=(\alpha_{k,i})_{k\in\Z,1\leq i\leq n_k} \hbox{ and }
\beta=(\beta_{k,i})_{k\in\Z,1\leq i\leq n_k}.$$
We assume that these parameters satisfy the equation
\begin{equation}
\label{eq-alphabeta}
\forall k\in\Z,\quad
\sum_{i=1}^{n_k}\alpha_{k,i}=\sum_{i=1}^{n_k}\beta_{k,i}=1.
\end{equation}
We will see the role of this equation in section \ref{section-gauss-map}.
The central values of these parameters are given by
$$\alpha_{k,i}^0=\beta_{k,i}^0=\frac{1}{n_k}.$$
Recall that the set $\{n_k\,:\,k\in\Z\}$ is finite, so 
the sequences $\alpha^0$ and $\beta^0$
are finitely valued as required.

If $||\alpha-\alpha^0||_{\infty}$ and $||\beta-\beta^0||_{\infty}$
are small enough, we have
$\alpha_{k,i}\neq 0$ and $\beta_{k,i}\neq 0$ for all 
$k\in\Z$ and $1\leq i\leq n_k$.
Then $g_k(z)\sim \frac{-\alpha_{k,i}}{z-a_{k,i}}$ in a neighborhood of $a_{k,i}$, so $\frac{1}{g_k}$
is a local complex coordinate in a neighborhood
of $a_{k,i}$.
In the same way, $g_{k+1}\sim\frac{\beta_{k,i}}{z-b_{k,i}}$
in a neighborhood of $b_{k,i}$, so $\frac{1}{g_{k+1}}$
is a local complex coordinate in a neighborhood of $b_{k,i}$.
The finiteness hypothesis allows us to find a
positive number $\rho$ so that if $a$, $b$, $\alpha$ and
$\beta$ are close enough to $a^0$, $b^0$, $\alpha^0$
and $\beta^0$ in $\ell^{\infty}$ norm, for any $k\in\Z$
and $1\leq i\leq n_k$, $\frac{1}{g_k}$ is a diffeomorphism
from a neighborhood $V_{k,i}$ of $a_{k,i}$ in $\C_k$
to the disk $D(0,\rho)$ and
$\frac{1}{g_{k+1}}$ is a diffeomorphism from a
neighborhood $W_{k,i}$ of $b_{k,i}$ in $\C_{k+1}$
to the disk $D(0,\rho)$.
We define
$$v_{k,i}:=\frac{1}{g_k}: V_{k,i}\stackrel{\sim}{\to} D(0,\rho)
\qquad v_{k,i}(a_{k,i})=0,$$
$$w_{k,i}:=\frac{1}{g_{k+1}} :
W_{k,i}\stackrel{\sim}{\to} D(0,\rho)
\qquad w_{k,i}(b_{k,i})=0.$$
By taking $\rho$ small enough, and still using the finiteness
hypothesis, we can assume that the ratio
$\left|\frac{v_{k,i}}{z-a_{k,i}}\right|$ in $V_{k,i}$
and $\left|\frac{w_{k,i}}{z-b_{k,i}}\right|$ in $W_{k,i}$
are bounded from above and below by some uniform
positive numbers (by which we mean that they are
independent of $k$, $i$ and the value of the parameters).
Hence these coordinates are admissible in the sense of definition 2 of \cite{TT}.

\medskip

We use these coordinates to open the nodes. Consider a real 
parameter $t \in (0,\rho) $.
For each $k\in\Z$ and $1\leq i\leq n_k$,
remove the disks $|v_{k,i}|\leq \frac{t^2}{\rho}$ from $V_{k,i}$
and $|w_{k,i}|\leq \frac{t^2}{\rho}$ from $W_{k,i}$.
Identify each point $z\in V_{k,i}$ with the point $z'\in W_{k,i}$
such that
$$v_{k,i}(z) w_{k,i}(z')=t^2.$$
This creates a neck connecting $\CC_k$ and $\CC_{k+1}$.
We call $\Sigma_t$ the resulting Riemann surface.
\subsection{The Gauss map}
\label{section-gauss-map}
We define a meromorphic function $g$ on $\Sigma_t$, to be the Gauss map, by
$$g(z)=(t g_k(z))^{(-1)^k}=
\left\{\begin{array}{l}
t g_k(z) \mbox{ if $z\in \CC_k$, $k$ even}\\
\frac{1}{t g_k(z)}\mbox{ if $z\in \CC_k$, $k$ odd}
\end{array}\right.$$
This function is well defined because if, say, $k$ is even and
$z\in V_{k,i}\subset\C_k$ is identified with $z'\in W_{k,i}\subset \C_{k+1}$, then
$$g(z)=t g_k(z)
=\frac{t}{v_{k,i}(z)}=\frac{w_{k,i}(z')}{t}=\frac{1}{t g_{k+1}(z')}
=g(z').$$
In a neighborhood of $\infty$ we have, thanks to the normalization
\eqref{eq-alphabeta}
$$g_k(z)\simeq \frac{1}{z^2}\left(
\sum_{i=1}^{n_{k-1}}\beta_{k-1,i} b_{k-1,i}
-\sum_{i=1}^{n_k} \alpha_{k,i} a_{k,i}
\right)$$
so the Gauss map has at least a double zero or pole at $\infty$,
as required for a planar end.
At the central value of the parameters we have
$$
\sum_{i=1}^{n_{k-1}}\beta_{k-1,i}^0 b_{k-1,i}^0
-\sum_{i=1}^{n_k} \alpha_{k,i}^0 a_{k,i}^0
=(-1)^k\conj^k\left(
\frac{1}{n_{k-1}}\sum_{i=1}^{n_{k-1}}p_{k-1,i}^0
-\frac{1}{n_k}\sum_{i=1}^{n_k}p_{k,i}^0
\right).$$
This is non-zero by hypothesis 3 of theorem \ref{th4},
so $g_k$ has a zero of multiplicity precisely
2 at $\infty_k$. 
By the finiteness hypothesis, this remains true when the parameters are close
to their central value in $\ell^{\infty}$ norm.
Hence the Gauss map has a double zero at $\infty_k$ if $k$ is even and a double
pole if $k$ is odd.
\subsection{The height differential}
\label{section-omega}
Fix some small number $\epsilon>0$.
For each $k\in\Z$, let $\Omega_{k,\epsilon}$ be $\C_k$ minus
the disks $D(a_{k,i}^0,\epsilon)$ for $1\leq i\leq n_k$ and
$D(b_{k-1,i}^0,\epsilon)$ for $1\leq i\leq n_{k-1}$.
Let $\Omega_{\epsilon}$ be the disjoint union of the domains $\Omega_{k,\epsilon}$ for
$k\in\Z$.
We assume that $||a-a^0||_{\infty}$, $||b-b^0||_{\infty}$ and $t$
are small enough so that the disks $|v_{k,i}|\leq\frac{t^2}{\rho}$
and $|w_{k,i}|\leq\frac{t^2}{\rho}$ (which were removed when defining
$\Sigma_t$) are included respectively in the disks
$D(a_{k,i}^0,\epsilon)$ and $D(b_{k,i}^0,\epsilon)$.
This allow us to see the fixed domain
$\Omega_{\epsilon}$ as a domain
in $\Sigma_t$.
Let $\Omega^1(\Sigma_t)$ be the space
of holomorphic 1-forms $\omega$ on $\Sigma_t$ such that the norm
$$||\omega||_{L^{\infty}(\Omega_{\epsilon})}
=\sup_{k\in\Z}\sup_{z\in \Omega_{k,\epsilon}}
\left|\frac{\omega}{dz}\right|$$
is finite.
This is well known to be a Banach space.

Next we define natural cycles on $\Sigma_t$.
For any $k\in\Z$ and $1\leq i\leq n_k$, let $A_{k,i}$ be the homology class
in $\Sigma_t$ of the circle $C(b_{k,i},\epsilon)$ in $\C_{k+1}$.
This circle is homologous in $\Sigma_t$ to
the circle $C(a_{k,i},\epsilon)$ with
the opposite orientation.
By theorem 2 in \cite{TT}, for $t$ small enough, the operator
$\omega\mapsto (\int_{A_{k,i}}\omega)_{k\in\Z,1\leq i\leq n_k}$ is
an isomorphism of Banach spaces from $\Omega^1(\Sigma_t)$ to
the set of sequences $\gamma=(\gamma_{k,i})_{k\in\Z,1\leq i\leq n_k}$ in
$\ell^{\infty}$ which satisfy the compatibility relation
\begin{equation}
\label{equation-compatibility}
\forall k\in\Z,\quad \sum_{i=1}^{n_k}\gamma_{k,i}=\sum_{i=1}^{n_{k-1}}\gamma_{k-1,i}.
\end{equation}
(Equation \eqref{equation-compatibility}
is what equation (2) of \cite{TT} becomes in our case.)
So we can define a holomorphic differential $\omega$ on $\Sigma_t$ by prescribing
its $A_{k,i}$-periods as
$$\int_{A_{k,i}}\omega=2\pi {\rm i}\gamma_{k,i},
\qquad k\in\Z,\, 1\leq i\leq n_k.$$
The new parameter is the sequence $\gamma=(\gamma_{k,i})_{k\in\Z,1\leq i\leq n_k}$.
The central value of this parameter is given by $\gamma_{k,i}^0=\frac{1}{n_k}$.
We require that
$$\forall k\in\Z,\quad\sum_{i=1}^{n_k}\gamma_{k,i}=1$$
so that the compatibility relation \eqref{equation-compatibility}
is satisfied.
\begin{proposition}
\label{proposition-omega}
The differential $\omega$ depends smoothly on
all parameters involved in this construction, namely
$a$, $b$, $\alpha$, $\beta$, $\gamma$ and $t$,
in a neighborhood of their respective central value
$a^0$, $b^0$, $\alpha^0$, $\beta^0$, $\gamma^0$ and $0$.
Moreover, when $t=0$, we have $\omega=\omega_k$ in $\C_k$,
where
$$\omega_k=
\sum_{i=1}^{n_{k-1}} \frac{\gamma_{k-1,i}}{z-b_{k-1,i}}dz
-\sum_{i=1}^{n_k}\frac{\gamma_{k,i}}{z-a_{k,i}}dz.$$
\end{proposition}
\begin{proof} The smoothness statement is theorem 4 in \cite{TT}.
Smoothness is for the norm defined above: 
in particular this means 
that the restriction of $\omega$ to the (fixed) domain $\Omega_{\epsilon}$ depends
smoothly on all parameters.

When $t=0$, $\omega$ is a regular differential on $\Sigma_0$ (see definition
1 in \cite{TT}) so has simple poles at all points
$a_{k,i}$, $b_{k,i}$.
The residues are determined by the prescribed periods, and this gives the
claimed formula.\end{proof}

\subsection{The equations we have to solve}
We define the Weierstrass data on $\Sigma_t$ by the standard formula
$$(\phi_1,\phi_2,\phi_3)=\left(\frac{1}{2}(g^{-1}-g)\omega,
\frac{{\rm i}}{2}(g^{-1}+g)\omega,\omega\right).$$
The minimal surface is defined by the Weierstrass representation formula :
$$\psi(z)=\Re\int_{z_0}^z (\phi_1,\phi_2,\phi_3) : \Sigma_t\to\R^3\cup\{\infty\}.$$
The points $\infty_k$ for $k\in\Z$ have to be the planar ends of our minimal surface.

For $\psi$ to be a regular immersion, we need that at each zero or pole
of the Gauss map $g$, which does not corresponds to an end, $\omega$ has a zero of the same order.
At the end $\infty_k$, the Gauss map has a double zero or pole.
To have an embedded planar end, 
we need that $\omega$ is holomorphic (which is already the case) and
does not vanish. We deal with these conditions in section \ref{section-zeros}.

Then we need $\psi(z)$ to be independent of the integration path from $z_0$ to $z$,
this is the period problem.
We fix some small number $\epsilon'>0$ and
we define the cycle $B_{k,i}$ for
$k\in\Z$ and $2\leq i\leq k$ as the composition of the following four paths :
\begin{enumerate}
\item a fixed path from $a_{k,i}^0+\epsilon'$ to $a_{k,1}^0+\epsilon'$,
\item a path from $a_{k,1}^0+\epsilon'$ to $b_{k,1}^0+\epsilon'$, going through the neck,
\item a fixed path from $b_{k,1}^0+\epsilon'$ to $b_{k,i}^0+\epsilon'$,
\item a path from $b_{k,i}^0+\epsilon'$ to $a_{k,i}^0+\epsilon'$, going through the neck.
\end{enumerate}
Note that unlike $A_{k,i}$,
the cycle $B_{k,i}$ is not defined when $t=0$.
We need to solve the following period problem :
$$\Re\int_{A_{k,i}}\phi_{\nu}=0\qquad
k\in\Z,\,1\leq i\leq n_k,\,1\leq \nu\leq 3,$$
$$\Re\int_{B_{k,i}}\phi_{\nu}=0\qquad
k\in\Z,\,2\leq i\leq n_k,\,1\leq \nu\leq 3.$$
Indeed, the first condition ensures that $\Re\int\phi_{\nu}$ is well defined
in each $\C_k$, and in particular that the residue at $\infty$ vanishes as required
for a planar end. The second condition ensures that $\Re\int\phi_{\nu}$
does not depend on the choice of the path from $\C_k$ to $\C_{k+1}$.

\medskip

Our strategy consists in proving that these equations, suitably
renormalized, extend smoothly at $t=0$, and in solving them
using the implicit function theorem at $t=0$ to determine
all parameters as functions of $t$.
\subsection{Zeros of $\omega$.}
\label{section-zeros}
Let us first locate the zeros and poles of $g$.
If $k$ is even, then the zeros of $g$ in $\C_k$ are the zeros of $g_k$, and $g$
has no poles in $\C_k$ (because the poles of $g_k$ were removed when opening nodes).
If $k$ is odd, the poles of $g$ in $\C_k$ are the zeros of $g_k$, and
$g$ has no zeros in $\C_k$.
What we need is that for each $k\in\Z$,
$\omega$ has a zero at each zero of $g_k$ in
$\C_k$, with the same multiplicity, and has no further zeros.
\begin{proposition}
\label{proposition-zeros}
For $(t,a,b,\gamma)$ in a neighborhood of $(0,a^0,b^0,\gamma^0)$,
there exist values of the parameters $\alpha$ and $\beta$, 
depending smoothly on $(t,a,b,\gamma)$, such that for 
all $k\in\Z$, $\omega$ has a zero at each finite zero of $g_k$ in $\C_k$,
with the same multiplicity, and has no further zero.
Moreover, when $t=0$,
we have $\alpha(0,a,b,\gamma)=\beta(0,a,b,\gamma)=\gamma$.
\end{proposition}
\begin{proof}
Since $g_k$ has a double zero at $\infty_k$, it has $n_k+n_{k-1}-2$ finite zeros
in $\C_k$, counting multiplicities. Let us count the number of zeros of
$\omega$ in $\C_k$.
By the finiteness hypothesis,
we may choose $\epsilon>0$ small enough
and $R$ large enough (both independent of $k$ and the parameters)
so that $U_k=\Omega_{k,\epsilon}\cap D(0,R)$
contains all zeros of $g_k$.
Let us write $\omega=f_k(z) dz$ in $\C_k$. The number of zeros of $\omega$
in $U_k$, counting multiplicities, is given by
$$N_k=\frac{1}{2\pi {\rm i}}\int_{\partial U_k}\frac{df_k}{f_k}.$$
At the central value of the parameters, we have $f_k=g_k$, so 
$f_k$ does not vanish on $\partial U_k$. Thanks to the
finiteness hypothesis, we may find a number $c>0$ independent of
$k$ such that for all $k$,
$|f_k|\geq c$ on $\partial U_k$ (still, at the central value of the
parameters).
By smooth dependence of $\omega$ on parameters, we have
$|f_k|\geq \frac{c}{2}$ on $\partial U_k$ when the parameters stay close,
in $\ell^{\infty}$ norm, to their central value.
Hence $N_k$ is a smooth, integer valued function of the parameters, so it is
constant. This proves that for each $k\in\Z$,
$\omega$ has $n_k+n_{k-1}-2$ zeros in $U_k$.

Let us now see that $\omega$ has no further zeros. It is proven in \cite{TT},
corollary 1, that there exists a uniform constant $C$ (independent
of $\epsilon$) such that the following is true 
for $t$ small enough : for any $k\in\Z$ and any $1\leq i\leq n_k$, if
$|\gamma_{k,i}|\geq C\epsilon ||\gamma||_{\infty}$, then $\omega$
has no zero in the annulus in $\Sigma_t$ bounded by the circles
$C(a_{k,i},\epsilon)$ in $\C_k$ and
$C(b_{k,i},\epsilon)$ in $\C_{k+1}$
(this annulus is what we call a neck).
Recall that the central value of the parameter
$\gamma_{k,i}$ is $\gamma_{k,i}^0=\frac{1}{n_k}$. Since the sequence
$(n_k)_{k\in\Z}$ is bounded, the ratio $\frac{|\gamma_{k,i}|}{||\gamma||_{\infty}}$
is bounded from below by a uniform positive number when the parameter
$\gamma$ is close to $\gamma^0$. Hence, provided we choose $\epsilon>0$
small enough, $\omega$ has no zeros on the necks.

It remains to consider the zeros of $\omega$ outside the disk $D(0,R)$ in 
$\C_k$. We introduce the coordinate $w=1/z$ on this domain and write
$\omega=h_k(w) dw$, $w\in D(0,\frac{1}{R})$.
At the central value of the parameters, the function $h_k$ has no zero, so $|h_k|$
is bounded from below by a constant independent of $k$ by the finiteness hypothesis.
Hence when the parameters are close to their central value, we have that for
any $k$, $\omega$ has no zero outside the disk $D(0,R)$ in $\C_k$.

Next we need to adjust the parameters $\alpha$ and $\beta$ so that $\omega$
vanishes at the zeros of $g_k$. One problem is that we do not know a priori
the multiplicities of the zeros of $g_k$. The following lemma is useful.
\begin{lemma}
\label{lemma-zeros}
Let $P$ be a polynomial of degree $d$ and $\Omega$ be a domain in $\C$ containing 
all the zeros of $P$.
Given a holomorphic function $f$ in $\Omega$,
let
$$F_i=\int_{\partial\Omega}\frac{z^i f(z)}{P(z)}dz.$$
If $F_i=0$ for all $0\leq i\leq d-1$ then $f/P$ is holomorphic in $\Omega$.
\end{lemma}
\noindent
{\it Proof of lemma \ref{lemma-zeros}}.
By the Weierstrass Preparation Theorem,
we may write $f=Ph+r$, where $h$ is holomorphic in $\Omega$ and $r$ is a polynomial
of degree less than $d$.
By Cauchy Theorem,
$$F_i=\int_{\partial\Omega}\frac{z^i r(z)}{P(z)}dz
=-2\pi {\rm i}
 \Res_{\infty} \left(\frac{ z^i r(z)}{P(z)}dz\right).$$
To establish last identity, we have used the Residue theorem on the complementary of
$\Omega$ in $\C$, and the fact that $\Omega$ contains all the zeros of $P$ so
the only pole is at infinity.
Assume that $r\neq 0$ and let $k=\deg(r)$. Take $i=d-k-1$. By a straightforward
computation, we get
$$\Res_{\infty} \frac{z^i r(z) dz}{P(z)}=-\frac{r_k}{p_d}$$
where $r_k$ and $p_d$ are respectively the leading coefficients of the polynomials
$r$ and $p$. Hence $F_i=0$ implies that $r_k=0$, contradicting the fact that $r$
has degree $k$. Hence $r=0$ so $f/P=h$ is holomorphic in $\Omega$.
\end{proof}

Returning to the proof of proposition \ref{proposition-zeros},
let
$$Q_k=\prod_{i=1}^{n_k} (z-a_{k,i})\prod_{i=1}^{n_{k-1}} (z-b_{k-1,i})$$
and write $g_k=\frac{P_k}{Q_k}$, so $P_k$ is a polynomial of degree
$n_k+n_{k-1}-2$.
Define, for $k\in\Z$ and $0\leq i\leq n_k+n_{k-1}-3$
$$\boZ_{k,i}=\int_{\partial U_k}\frac{z^i\omega}{P_k(z)}.$$
Let $\boZ=\boZ(t,a,b,\alpha,\beta,\gamma)=(\boZ_{k,i})_{k\in\Z,0\leq i\leq n_k+n_{k-1}-3}$.
This is a smooth function of all parameters (with the $\ell^{\infty}$ norm
on the target space).
If $\boZ=0$, then by the lemma, for all $k\in\Z$,
$\omega/P_k$ is holomorphic in $U_k$, so at each zero of $g_k$, $\omega$
has a zero with at least same order. 
Since we counted the number of zeros to
be the same, the multiplicities are equal as required.
The proposition then follows from the following lemma and the implicit
function theorem.
(The last statement of the proposition follows from the uniqueness part
of the implicit function theorem.)
\begin{lemma}
\label{lemma-zeros2}
When $t=0$ and $\alpha=\beta=\gamma$, we have
$\boZ(0,a,b,\gamma,\gamma,\gamma)=0$. Moreover, the partial differential
of $\boZ$ with respect to $(\alpha,\beta)$ at $(0,a^0,b^0,\alpha^0,\beta^0,\gamma^0)$
is an isomorphism of Banach spaces.
\end{lemma}
\noindent
{\it Proof of lemma \ref{lemma-zeros2}}.
If $t=0$ and $\alpha=\beta=\gamma$, then by proposition \ref{proposition-omega},
$\omega=g_k dz$ in $\C_k$ for all $k\in\Z$, so $\boZ=0$.
We write $D$ for the partial differential with respect to $(\alpha,\beta)$ at
the central value of the parameters.
Observe that when $t=0$, $\Sigma_0$ does not depend on $(\alpha,\beta)$,
so $\omega$ does not depend on $(\alpha,\beta)$ and $D\omega=0$.
Hence
$$D\left(\frac{z^i\omega}{P_k}\right)=
-\frac{z^i\omega}{P_k^2}DP_k
=-\frac{z^idz}{P_k Q_k} DP_k
=-\frac{z^idz}{P_k}D g_k$$
$$D \boZ_{k,i}(\alpha,\beta)=
-
\int_{\partial U_k}\frac{z^i}{P_k} f_k dz
\quad \mbox{ with }
f_k=
\sum_{j=1}^{n_{k-1}}\frac{\beta_{k-1,j}}{z-b_{k-1,j}^0}
-\sum_{j=1}^{n_k}\frac{\alpha_{k,j}}{z-a_{k,j}^0}.  $$
From this we see that $D \boZ_k(\alpha,\beta)$ only depends
on $\alpha_k$ and $\beta_{k-1}$ 
(so $D\boZ$ has ``block diagonal'' form).
Let us prove that for each $k$,
the operator $(\alpha_k,\beta_{k-1})
\mapsto D\boZ_k(\alpha_k,\beta_{k-1})$ is
an isomorphism. The domain and target spaces both have finite
complex dimension
$n_k+n_{k-1}-2$, because by normalization of the parameters,
we have $\sum_{i=1}^{n_k}\alpha_{k,i}=\sum_{i=1}^{n_{k-1}}\beta_{k-1,i}=0$ (instead of $1$
since we are in the tangent space to the parameter space when we
compute the differential).
Let $(\alpha_k,\beta_{k-1})$ be in the kernel. 
Then by lemma \ref{lemma-zeros}, the function $\frac{f_k}{P_k}$ is holomorphic
in $U_k$.
Hence the polynomials $f_k Q_k$ and $P_k$, which have the same degree,
have the same zeros, hence are proportional.
So there exists $\lambda\in\C$ such that
$\alpha_k=\lambda \alpha_k^0$ and
$\beta_{k-1}=\lambda\beta_{k-1}^0$.
Because of the normalizations this gives $\alpha_k=\beta_{k-1}=0$.
Hence for each $k$, the operator $(\alpha_k,\beta_{k-1})\mapsto
D\boZ_k(\alpha_k,\beta_{k-1})$ is an isomorphism. By the finiteness hypothesis,
the inverse of these operators is bounded by a constant independent of $k$.
It readily follows that $D\boZ$ is an isomorphism of Banach spaces from
$\ell^{\infty}$ to $\ell^{\infty}$.
\cqfd
\subsection{The period problem for $\omega$}
\begin{proposition}
\label{proposition-period-omega}
Assume that $\alpha$ and $\beta$ are given by proposition
\ref{proposition-zeros}.
For $(t,a,b)$ in a neighborhood of $(0,a^0,b^0)$, there exist values
of the parameter $\gamma$, depending continuously on
$(t,a,b)$, such that the following period problem is solved :
$$\Re\int_{A_{k,i}}\omega=0,\qquad k\in\Z,\,1\leq i\leq n_k,$$
$$\Re\int_{B_{k,i}}\omega=0,\qquad k\in\Z,\,2\leq i\leq n_k.$$
Moreover, when $t=0$, we have
$\gamma_{k,i}(0,a,b)=\frac{1}{n_k}$.
\end{proposition}
\begin{proof}
The period problem for the cycles $A_{k,i}$ is equivalent to $\gamma_{k,i}\in\R$,
which we assume from now on.
Regarding the cycles $B_{k,i}$, we need the following
\begin{lemma}
\label{lemma-period-omega}
The function
$$\left(\Re\int_{B_{k,i}}\omega
\; - 2(\gamma_{k,i}- \gamma_{k,1})\log t\right)_{k\in\Z,2\leq i\leq n_k}$$
extends smoothly at $t=0$ to a smooth function of all parameters
(with the $\ell^{\infty}$ norm on the target space).
\end{lemma}
The proof of this lemma is technical and given in appendix \ref{appendix}.
We make the change of variable
$t=e^{-1/\tau^2}$, where $\tau>0$ is in a neighborhood of $0$.
Define
$$\boV_{k,i}=\tau^2\Re\int_{B_{k,i}}\omega$$
and $\boV=(\boV_{k,i})_{k\in\Z,2\leq i\leq n_k}$.
By the lemma, the function $\boV$ extends smoothly at $\tau=0$ to a smooth function
of the parameters $(\tau,a,b,\gamma)$.
Moreover, when $\tau=0$, we have
$$\boV(0,a,b,\gamma)=-2(\gamma_{k,i}-\gamma_{k,1}).$$
For each $k\in\Z$, the partial differential of $\boV_k$ with respect to $\gamma$
only depends on $\gamma_k$, and it is straightforward to see that it is an
isomorphism. 
(The domain and target spaces are both real vector spaces of dimension
$n_k-1$ :
recall the normalization $\sum_{i=1}^{n_k}\gamma_{k,i}=1$.)
The proposition then follows by the implicit function theorem.
\end{proof}
\begin{remark}
$\gamma$ is a smooth function of the parameters $(\tau,a,b)$.
From now on the parameter $\tau$ replaces the parameter $t$.
\end{remark}
\subsection{The $B$-period problem for $\phi_1$ and $\phi_2$.}
In this section we solve the period problem
$$\Re\int_{B_{k,i}}\phi_1=\Re\int_{B_{k,i}}\phi_2=0,
\qquad k\in\Z,\,2\leq i\leq n_k.$$
This is equivalent to
$$\int_{B_{k,i}}g^{-1}\omega=\overline{\int_{B_{k,i}}g\omega},
\qquad k\in\Z,\,2\leq i\leq n_k.$$
We define for $k\in\Z$ and $2\leq i\leq n_k$
$$\boH_{k,i}= t\left(
\int_{B_{k,i}} g^{(-1)^k}\omega -
\overline{\int_{B_{k,i}} g^{(-1)^{k+1}}\omega}\right).$$
Let $\boH=(\boH_{k,i})_{k\in\Z,2\leq i\leq n_k}$.
We want to solve the equation $\boH=0$.
\begin{lemma}
\label{lemma-period-horizontal}
The function $\boH$ extends smoothly at $\tau=0$ to a smooth function
of all parameters (for the $\ell^{\infty}$ norms).
Moreover, when $\tau=0$ we have
$$\boH_{k,i}=
\int_{b_{k,1}}^{b_{k,i}}g_{k+1}^{-1}\omega_{k+1}
-\overline{\int_{a_{k,i}}^{a_{k,1}}g_k^{-1}\omega_k}.$$
\end{lemma}
The proof of this lemma is technical
and is given in appendix \ref{appendix}.

\medskip

Assume that the parameters $\alpha$, $\beta$ and $\gamma$ are
determined as functions of $(\tau,a,b)$ by propositions
\ref{proposition-zeros} and \ref{proposition-period-omega}.
Then $\boH$ extends at $\tau=0$ to a smooth function
of $(\tau,a,b)$. Moreover, as $\alpha=\beta=\gamma$ when
$\tau=0$, we have $\omega_k=g_k dz$ so
$$\boH_{k,i}(0,a,b)=
b_{k,i}-b_{k,1}+\overline{a_{k,i}}-\overline{a_{k,1}}.$$
We normalize the parameter $b$ by requiring that
$b_{k,1}=0$ for all $k\in\Z$.
(This may be seen as a normalization of translation in $\C_k$.)
Let $\boH_k=(\boH_{k,i})_{2\leq i\leq n_k}$.
For each $k$, the partial differential of $\boH_k$ with respect to
$b$ only depends on $b_k$, and is easily seen to be injective, so 
is an isomorphism because the domain and range both have complex dimension
$n_k-1$.
Hence, by the finiteness hypothesis, the partial differential of $\boH$ with
respect to $b$ is an isomorphism from $\ell^{\infty}$ to $\ell^{\infty}$.
By the implicit function theorem, we get
\begin{proposition}
\label{proposition-period-horizontal}
Assume that the parameters $\alpha,\beta,\gamma$ are determined by proposition
\ref{proposition-zeros} and \ref{proposition-period-omega}.
For $(\tau,a)$ in a neighborhood of $(0,a^0)$, there exist values of the
parameter $b$, depending smoothly on $(\tau,a)$, such that
$\boH(\tau,a,b(\tau,a))=0$. Moreover, when $\tau=0$, we have
$$b_{k,i}=-\overline{a_{k,i}}+\overline{a_{k,1}}.$$
\end{proposition}
\subsection{The $A$-period problem for $\phi_1$ and $\phi_2$}
In this section we solve the period problem
$$\Re\int_{A_{k,i}}\phi_1=\Re\int_{A_{k,i}}\phi_2=0,
\qquad k\in\Z,\,1\leq i\leq n_k.$$
This is equivalent to
$$\int_{A_{k,i}}g^{-1}\omega=\overline{\int_{A_{k,i}} g\omega},
\qquad k\in\Z,\,1\leq i\leq n_k.$$
Recall that $\conj(z)=\overline{z}$ denotes the
 complex conjugation.
We define
$$\boF_{k,i}^-=\frac{-1}{t}\conj^k\left(\int_{A_{k,i}} g^{(-1)^k}\omega\right),$$
$$\boF_{k,i}^+=\frac{1}{t}\conj^{k+1}\left(\int_{A_{k,i}} g^{(-1)^{k+1}}\omega
\right),$$
$$\boF_{k,i}=\boF_{k,i}^- + \boF_{k,i}^+.$$
We want to solve the equations $\boF_{k,i}=0$ for all $k\in\Z$ and $1\leq i\leq n_k$.
This equation will give us the balancing equation $F_{k,i}=0$ of
section \ref{section-configurations}.
However, to be able to use the non-degeneracy hypothesis, we have to reformulate
this infinite system of equations in a slightly different way,
as we did for the balancing condition
$F_{k,i}=0$ in section \ref{section-configurations}.
Define
$$\boF_k^-=\sum_{i=1}^{n_k} \boF_{k,i}^-,$$
$$\boF_k^+=\sum_{i=1}^{n_k} \boF_{k,i}^+,$$
$$\boF_k=\sum_{i=1}^{n_k}\boF_{k,i}=\boF_k^-+\boF_k^+.$$
\begin{lemma}
It holds, independently of the values of the parameters,
$$\forall k\in\Z,\qquad \boF_k^-+\boF_{k-1}^+=0.$$
\end{lemma}
\begin{proof}
We have, in $\C_k$, $g^{(-1)^k}=t g_k$. This function is
holomorphic at $\infty_k$. Hence, by the residue theorem,
$$-\sum_{i=1}^{n_k}\int_{A_{k,i}} g^{(-1)^k}\omega
+\sum_{i=1}^{n_{k-1}}\int_{A_{k-1,i}} g^{(-1)^k}\omega=0$$
so
$$t\conj^k(\boF_k^-)+t\conj^k(\boF_{k-1}^+)=0.$$
which proves the lemma.
\end{proof}

Let us write $\boG_k=-\boF_k^-$.
Then by the lemma,
$$\boF_k=\boF_k^- - \boF_{k+1}^- =\boG_{k+1}-\boG_k.$$
Solving $\boF_{k,i}=0$ for all $k\in\Z$ and
$1\leq i\leq n_k$ is then equivalent to solve for all $k\in\Z$,
$\boF_{k,i}=0$ for $2\leq i\leq n_k$ and $\boG_k=\boG_0$.
We write $\boF=(\boF_{k,i})_{k\in\Z,2\leq i\leq n_{k}}$
and $\boG=(\boG_k)_{k\in\Z}$.

Next we need to introduce the parameters $\ell_k$ and $u_{k,i}$ of section
\ref{section-configurations}
to make use of the non-degeneracy hypothesis. We make the change of parameter
$$a_{k,i}=(-1)^k \conj^k( \ell_k+u_{k,i}).$$
We define the parameter ${\bf U}$ by equation \eqref{u}.
\begin{proposition}
\label{proposition-balancing}
Assume that the parameters $\alpha$, $\beta$, $\gamma$ and
$b$ are determined by propositions \ref{proposition-zeros},
\ref{proposition-period-omega} and \ref{proposition-period-horizontal}.
The functions $\boF$ and $\boG$ extend at $\tau=0$ to smooth functions of
$(\tau,{\bf U})$ (for the $\ell^{\infty}$ norms).
Moreover, at $\tau=0$ we have
$$\boF_{k,i}(0,{\bf U})=4\pi {\rm i} F_{k,i}({\bf U}),$$
$$\boG_k(0,{\bf U})=4\pi {\rm i} G_k({\bf U})$$
where the functions $F_{k,i}$ and $G_k$ are as in section
\ref{section-configurations}.
Hence, if the configuration is balanced and non-degenerate in the sense of definition \ref{def.iso},
for $\tau$ in a neighborhood of $0$, there exist values of
the parameter ${\bf U}$, depending smoothly on $\tau$ for the $\ell^{\infty}$
norm, such that $\boF({\bf U}(\tau))=\boF({\bf U}^0)=0$ and
$\boG({\bf U}(\tau))=\boG({\bf U}^0)=\mbox{constant}$, so our period problem is solved.
\end{proposition}
\begin{proof}
We have $g^{(-1)^k}=t g_k$ in $\C_k$, and $A_{k,i}$ is homologous
to the circle $C(a_{k,i},\epsilon)$ in $\C_k$
with the negative orientation, so
$$\boF_{k,i}^-=\conj^k\left(\int_{C(a_{k,i},\epsilon)}g_k\omega\right).$$
We see by this formula that $\boF_{k,i}^-$ extends at $\tau=0$.
Now $\omega$ depends smoothly on all parameters (for the norm $||\omega||_{\infty}$
defined in section \ref{section-omega}), and the function 
equal to
$g_k$ in each $\C_k$ also depends smoothly on all parameters (using the
finiteness hypothesis). So $(\boF_{k,i}^-)_{k\in\Z,1\leq i\leq n_k}$ is
a smooth function of all parameters by composition with a bounded linear
operator.

In the same way, we have $g^{(-1)^{k+1}}=t g_{k+1}$ in $\C_{k+1}$
and $A_{k,i}$ is homologous to the circle $C(b_{k,i},\epsilon)$ in $\C_{k+1}$,
so
$$\boF_{k,i}^+=\conj^{k+1}\left(\int_{C(b_{k,i},\epsilon)}g_{k+1}\omega\right).$$
It follows that
$(\boF_{k,i}^+)_{k\in\Z,1\leq i\leq n_k}$ is a smooth function of all parameters.
Hence $\boF$ and $\boG$ are smooth functions of the parameters.

Next assume that $\tau=0$
and $\alpha$, $\beta$, $\gamma$ and $b$ are determined by
propositions \ref{proposition-zeros},
\ref{proposition-period-omega} and \ref{proposition-period-horizontal}.
Then
$$\alpha_{k,i}=\beta_{k,i} =\gamma_{k,i}=\frac{1}{n_k}=c_k,$$
and $\omega=g_k dz$ in $\C_k$. Also
$$b_{k,i}=(-1)^{k+1}\conj^{k+1}(u_{k,i}).$$
Then we compute $\boF_{k,i}^-$ and $\boF_{k,i}^+$ as residues :
\begin{eqnarray*}
\boF_{k,i}^- &=& \conj^k\left(2\pi {\rm i}
\Res_{a_{k,i}} (g_k)^2\right)\\
&=& 2\pi {\rm i} (-1)^k \conj^k\left( 2\sum_{j\neq i}\frac{c_k^2}{a_{k,i}-a_{k,j}}
-2\sum_{j=1}^{n_{k-1}}\frac{c_k c_{k-1}}{a_{k,i}-b_{k-1,j}}\right)
\\
&=& 4\pi {\rm i}\left(\sum_{j\neq i}\frac{c_k^2}{u_{k,i}-u_{k,j}}
-\sum_{j=1}^{n_{k-1}} \frac{c_k c_{k-1}}{\ell_k+u_{k,i}-u_{k-1,j}}\right)
\end{eqnarray*}
\begin{eqnarray*}
\boF_{k,i}^+ &=& \conj^{k+1}\left(2\pi {\rm i}\Res_{b_{k,i}} (g_{k+1})^2\right)\\
&=& 2\pi {\rm i} (-1)^{k+1} \conj^{k+1}\left( 2\sum_{j\neq i}\frac{c_k^2}{b_{k,i}-b_{k,j}}
-2\sum_{j=1}^{n_{k+1}}\frac{c_k c_{k+1}}{b_{k,i}-a_{k+1,j}}\right)
\\
&=& 4\pi {\rm i}\left(\sum_{j\neq i}\frac{c_k^2}{u_{k,i}-u_{k,j}}
-\sum_{j=1}^{n_{k+1}} \frac{c_k c_{k+1}}{u_{k,i}-\ell_{k+1}-u_{k+1,j}}\right)
\end{eqnarray*}
This gives $\boF_{k,i}=4\pi i F_{k,i}$.
Regarding $\boG_k$, we have
$$\sum_{i=1}^{n_k}\sum_{j=1\atop j\neq i}^{n_k} \frac{1}{u_{k,i}-u_{k,j}}=0$$
so $\boG_k=4\pi i G_k$.
\end{proof}
\subsection{Embeddedness}
\label{section-embedded}
At this point we have constructed a one parameter family of
minimal immersion $\psi_{t}:\Sigma_{t}\to\R^3$
for $t>0$ small enough.
(We switch back to the parameter $t$.)
Let $M_t=\psi_{t}(\Sigma_t)$.
\begin{proposition}
\label{proposition-embedded}
$M_t$ is an embedded minimal surface.
\end{proposition}
\begin{proof}
Fix a complex number $O$ and for $k\in\Z$,
let $O_k$ be the point $z=O$ in $\C_k$.
By the finiteness hypothesis, we may
choose $O$ so that $O_k\in\Omega_{k,\epsilon}$
for all $k\in\Z$.
Let $\varphi_t:\R^3\to\R^3$ be the affine transformation
$(x_1,x_2,x_3)\mapsto (2tx_1,2tx_2,x_3)$.
Define
$$\psi_{k,t}(z)=\varphi_t(\psi_t(z)-\psi_t(O_k)):
\Omega_{k,\epsilon}\to\R^3.$$
If $k$ is even, then $g=t g_k$ in $\C_k$ so
$$\psi_{k,t}(z)=\Re\int_{O}^z \left(
(g_k^{-1}-t^2 g_k)\omega,
{\rm i}(g_k^{-1}+t^2 g_k)\omega,
\omega\right).$$
This extends (as a smooth function of $\tau$) at $t=0$, with
value
$$\psi_{k,0}(z)=\Re\int_{O}^z\left(dz,{\rm i}dz,\omega_k\right)
=\left(\Re(z),-\Im(z),h_k(z)-h_k(O)\right)$$
where the function $h_k$ is defined by
$$h_k(z)=\sum_{i=1}^{n_{k-1}}c_{k-1}\log|z-b_{k-1,i}^0|
-\sum_{i=1}^{n_k}c_k \log|z-a_{k,i}^0|.$$
Hence the image $\psi_{k,t}(\Omega_{k,\epsilon})$ converges
to the graph of $z\mapsto h_k(\overline{z})-h_k(\overline{O})$
on $\Omega_{k,\epsilon}$,
so it is embedded for $t$ small enough.
When $k$ is odd, we have by similar computations that
$\psi_{k,t}(\Omega_{k,\epsilon})$ converges to
the graph of $z\mapsto h_k(-z)-h_k(-O)$.

\medskip

Observe that the function $h_k$ is bounded on $\Omega_k$.
Provided $\epsilon$ is small enough, we can find a number
$\eta$ large enough such that the level line
 $h_k=\eta$ consists of
$n_k$ closed convex curves around the points $a_{k,i}^0$,
$1\leq i\leq n_k$, and the level line $h_k=-\eta$
consists of $n_{k-1}$ closed convex curves around the points
$b_{k-1,i}^0$, $1\leq i\leq n_{k-1}$.

\medskip

Back to our minimal immersion $\psi_t$, let
$M_{k,t}$ be the intersection of the image
$\psi_t(\Omega_{k,\epsilon})$
with the slab $\psi_t(O_k)-\eta< x_3<\psi_t(O_k)+\eta$.
By what we have seen, for $t$ small enough, each $M_{k,t}$
is embedded, and its boundary consists of $n_k$ closed
horizontal convex curves on the top and $n_{k-1}$
closed convex curves on the bottom.
By lemma \ref{lemma-appendix2}, $\int_{O_k}^{O_{k+1}}\omega\sim -2c_k \log t$,
so we see that $M_{k+1,t}$ lies strictly above $M_{k,t}$.

\medskip

The intersection of $M_t$
with the slab $\psi_t(O_k)+\eta<x_3<\psi_t(O_{k+1})-\eta$ is
the union of $n_k$ minimal annuli. Each annulus is bounded
by convex curves in horizontal planes, so is foliated
by horizontal convex curves by a theorem of Schiffman
\cite{S}, hence embedded.

\medskip

It remains to see that these $n_k$ annuli are disjoint.
Consider one of these annuli.
There exists $i$, $1\leq i\leq n_k$ such that
our annulus
is included in the image of the annulus bounded by the circles
$C(a_{k,i}^0,\epsilon')$ and $C(b_{k,i}^0,\epsilon')$
(provided we take $\eta$ large enough).
The image of these circles are close to circles of radius
$\frac{\epsilon'}{2t}$.
By lemma \ref{lemma-appendix3},
$t\int_{a_{k,i}^0+\epsilon'}^{b_{k,i}^0+\epsilon'} g^{\pm 1}\omega$
extends smoothly at $t=0$ (as a smooth function of $\tau$),
with value $\pm \epsilon'$.
Hence the boundary circles are inside a vertical cylinder of
radius $\frac{2\epsilon'}{t}$ for $t$ small enough.
By the convex hull property of minimal surfaces,
the annulus is inside this cylinder.

\medskip

Now from our analysis of $\psi_{k,t}$,
the axes of these cylinders are separated by a distance
greater than $\frac{c}{t}$ for some uniform positive number $c$.
Hence these cylinders are disjoint provided we take $\epsilon'>0$
small enough.
This proves that the annuli are disjoint, so $M_t$ is embedded.
\end{proof}
\appendix
\section{Proof of lemmas \ref{lemma-period-omega} and 
\ref{lemma-period-horizontal}}
\label{appendix}
We start by proving lemma \ref{lemma-period-omega}.
The period of $\omega$ on $B_{k,i}$ has four terms corresponding to the
four paths in the definition of $B_{k,i}$.
The first term is easily dealt with.
Indeed, the path from $a_{k,i}^0+\epsilon'$ to $a_{k,1}^0+\epsilon'$
is fixed and may be chosen in the domain $\Omega_{k,\epsilon'}$.
The restriction of $\omega$ to $\Omega_{\epsilon'}$ depends smoothly on
all parameters, and we compose it with the linear operator
$\omega\mapsto(\int_{a_{k,i}^0+\epsilon'}^{a_{k,1}^0+\epsilon'}\omega)_{k\in\Z,
2\leq i\leq n_k}$ which is bounded from $L^{\infty}(\Omega_{\epsilon'})$
to $\ell^{\infty}$ hence smooth.
We treat the third term, where the integral along a path
joining $b_{k,1}^0 +\epsilon'$ to $b_{k,i}^0+\epsilon'$
appears, in the same way.

To handle the second and fourth terms, we expand $\omega$ in Laurent series
and estimate carefully its coefficients.
First let us  define various constants and in particular
explain how we choose $\epsilon'$ in the definition of
the cycle $B_{k,i}$.
Remember that we have fixed a small number $\epsilon>0$.
We choose a constant $r$ independent of $k$
such that for all $k\in\Z$, $|g_k|\leq \frac{r}{2}$
in $\Omega_{k,\epsilon}$.
We choose a small number $\epsilon'>0$ independent of $k$
such that for all $k\in\Z$ and
$1\leq i\leq n_k$,
$|g_k|\geq 2r$ in the disk $D(a_{k,i},2\epsilon')$
and $|g_{k+1}|\geq 2r$ in the disk $D(b_{k,i},2\epsilon')$.
If $||a-a^0||_{\infty}\leq\epsilon'$ and $||b-b^0||_{\infty}\leq\epsilon'$, we
have $|g_k|\geq 2r$ in the disk $D(a_{k,i}^0,\epsilon')$
and $|g_{k+1}|\geq 2r$ in the disk $D(b_{k,i}^0,\epsilon')$.
Finally, we choose a constant $r'$ independent of $k$ such that
for all $k\in\Z$,
$|g_k|\leq \frac{r'}{2}$ in $\Omega_{k,\epsilon'}$.
All this is possible thanks to the finiteness hypothesis.
\begin{lemma}
\label{lemma-appendix1}
Let $v=v_{k,i}$. We have, in the annulus
$\frac{t^2}{\rho}\leq |v|\leq\rho$,
$$\omega=-\gamma_{k,i}\frac{dv}{v}+\sum_{n\geq 0}r^{n+1} c_{k,i,n}^+ v^n dv
+\sum_{n\geq 2} (rt^2)^{n-1} c_{k,i,n}^- \frac{dv}{v^n}.$$
The coefficients $c_{k,i,n}^{\pm}$ are bounded.
More precisely,
let $c^+=(c_{k,i,n}^+)_{k\in\Z,1\leq i\leq n_k,n\geq 0}$ and
$c^-=(c_{k,i,n}^-)_{k\in\Z,1\leq i\leq n_k,n\geq 2}$.
Then $c^+$ and $c^-$ are in $\ell^{\infty}$ and are smooth functions of all
parameters.
\end{lemma}
\begin{proof}
Using $v=v_{k,i}$ as a coordinate, we can write the Laurent series of $\omega$ 
in the annulus $\frac{t^2}{\rho}\leq |v|\leq\rho$ as
$$\omega=\sum_{n\in\Z} c_{k,i,n}v^n dv.$$
The coefficient $c_{k,i,n}$ is given by
$$c_{k,i,n}=
\frac{1}{2\pi {\rm i}}
\int_{\partial V_{k,i}}\frac{\omega}{v^{n+1}}
= \frac{1}{2\pi {\rm i}}\int_{\partial V_{k,i}}\omega g_k^{n+1}.$$
In particular $c_{k,i,-1}=-\gamma_{k,i}$.
For $n\geq 0$, we write $c_{k,i,n}=r^{n+1} c_{k,i,n}^+$ with
$$c_{k,i,n}^+=\frac{1}{2\pi {\rm i}}
\int_{C(a_{k,i}^0,\epsilon)} \omega \left(\frac{g_k}{r}\right)^{n+1}.$$
We have used that $\partial V_{k,i}$ is homologous to the circle $C(a_{k,i}^0,\epsilon)$.
We have $|\frac{g_k}{r}|\leq \frac{1}{2}$ in $\Omega_{k,\epsilon}$
so $(c_{k,i,n}^+)_{n\in\N}$ is bounded.
The function equal to $\frac{g_k}{r}$
in each $\Omega_{k,\epsilon}$ is in the open
unit ball of $L^{\infty}(\Omega_{\epsilon})$ and depends smoothly on
parameters.
Recall that $L^{\infty}(\Omega_{\epsilon})$ is a Banach algebra for the
pointwise product.
If $\boA$ is a Banach algebra,
the map $x\mapsto (x^n)_{n\in\N}$ is smooth from the open
unit ball of $\boA$ to $\boA^{\N}$ with the sup norm
(an easy exercice).
From this and the fact that $\omega$, restricted to $\Omega_{\epsilon}$,
depends smoothly on parameters, and composition with a bounded linear
operator, we conclude that $c^+$ depends smoothly on all parameters.

For $n\leq -2$ we write $m=-n\geq 2$
and $c_{k,i,-m}=(rt^2)^{m-1} c_{k,i,m}^-$ with
\begin{eqnarray*}
c_{k,i,m}^-&=&\frac{1}{2\pi {\rm i}}\int_{\partial V_{k,i}}
\omega\left(\frac{v_{k,i}}{rt^2}\right)^{m-1}
=\frac{-1}{2\pi {\rm i}}
\int_{\partial W_{k,i}} \omega \left(\frac{1}{r w_{k,i}}
\right)^{m-1}\\
&=&\frac{-1}{2\pi {\rm i}}\int_{C(b_{k,i}^0,\epsilon)}
\omega\left(\frac{g_{k+1}}{r}\right)^{m-1}.
\end{eqnarray*}
We have used the fact that $v_{k,i}w_{k,i}=t^2$ and
$\partial V_{k,i}$ is homologous to $-\partial W_{k,i}$.
From this we conclude in the same way as above that $c^-$ is in $\ell^{\infty}$
and depends smoothly on all parameters.
\end{proof}

Lemma \ref{lemma-period-omega} follows from the following lemma.
\begin{lemma}
\label{lemma-appendix2}
The function
$$\left(\int_{a_{k,i}^0+\epsilon'}^{b_{k,i}^0+\epsilon'}\omega
+ 2\gamma_{k,i}\log t\right)_{k\in\Z,1\leq 
i\leq n_k}$$ extends smoothly at $t=0$ to a smooth function of all parameters
(for the $\ell^{\infty}$ norm).
\end{lemma}
\begin{proof} Let $\varphi_{k,i}=v_{k,i}(a_{k,i}+\epsilon')$ and
$\psi_{k,i}=w_{k,i}(b_{k,i}+\epsilon')$, so
$v_{k,i}(b_{k,i}+\epsilon')=\frac{t^2}{\psi_{k,i}}$.
Then by lemma \ref{lemma-appendix1}
\begin{eqnarray*}
\int_{a_{k,i}^0+\epsilon'}^{b_{k,i}^0+\epsilon'}\omega
&=&\int_{v=\varphi_{k,i}}^{\frac{t^2}{\psi_{k,i}}}
\left(-\gamma_{k,i}\frac{dv}{v}+
\sum_{n\geq 0}r^{n+1} c_{k,i,n}^+ v^n dv
+\sum_{n\geq 2} (rt^2)^{n-1} c_{k,i,n}^- \frac{dv}{v^n}\right)\\
&=& -\gamma_{k,i}\log t^2 +\gamma_{k,i}\log(\varphi_{k,i}\psi_{k,i})\\
&+&\sum_{n\geq 0}\frac{c_{k,i,n}^+}{n+1}\left(\left(\frac{rt^2}{\psi_{k,i}}\right)^{n+1}
-(r\varphi_{k,i})^{n+1}\right)\\
&+&\sum_{n\geq 2}\frac{c_{k,i,n}^-}{n-1}\left(\left(\frac{rt^2}{\varphi_{k,i}}\right)^{n-1}
-(r\psi_{k,i})^{n-1}\right)
\end{eqnarray*}
By our choice of $\epsilon'$ we have $|r\varphi_{k,i}|\leq\frac{1}{2}$
and $|r\psi_{k,i}|\leq \frac{1}{2}$.
By our choice of $r'$ we have $|\frac{rt^2}{\varphi_{k,i}}|\leq \frac{rr't^2}{2}
\leq \frac{1}{2}$ provided $t^2\leq \frac{1}{rr'}$. 
In the same way $|\frac{rt^2}{\psi_{k,i}}|\leq\frac{1}{2}$.
So the sequences $(r\varphi_{k,i})_{k,i}$, $(r\psi_{k,i})_{k,i}$,
$(\frac{rt^2}{\varphi_{k,i}})_{k,i}$ and
$(\frac{rt^2}{\psi_{k,i}})_{k,i}$ are all in the ball of radius
$\frac{1}{2}$ in $\ell^{\infty}$, and they depend smoothly on all parameters
(they are given by explicit formula).
Let us deal only with the term in the above formula
containing $r\varphi_{k,i}$, as the others are
similar.
We write this term as
$$\sum_{n\geq 0}\frac{c_{k,i,n}^+}{n+1}\left(\frac{1}{\sqrt{2}}\right)^{n+1}
(r\sqrt{2}\varphi_{k,i})^{n+1}.$$
The sequence $(r\sqrt{2}\varphi_{k,i})_{k,i}$ is in the open unit ball
of $\ell^{\infty},$ so
by the fact recalled above about Banach algebras,
the sequence $((r\sqrt{2}\varphi_{k,i})^{n+1})_{k,i,n}$
is in $\ell^{\infty}$ and depends smoothly on parameters.
We multiply by the sequence $(c_{k,i,n}^+)_{k,i,n}$ which depends smoothly
on parameters, and compose with the linear operator which maps the
sequence $(x_{k,i,n})_{k\in\Z,1\leq i\leq n_k,n\in\N}$ in $\ell^{\infty}$ to
$$\left(\sum_{n\geq 0}\frac{1}{n+1}\left(\frac{1}{\sqrt{2}}\right)^{n+1} x_{k,i,n}
\right)_{k\in\Z,1\leq i\leq n_k}.$$
This operator is bounded,
so we conclude that third term in the  expansion
of the integral of $\omega$ depends smoothly on parameters.
We deal with the other terms in the same way.
\end{proof}

To prove lemma \ref{lemma-period-horizontal}, we need a lemma
which is similar to lemma \ref{lemma-appendix2}.
\begin{lemma}
\label{lemma-appendix3}
The function
$$\left(\int_{a_{k,i}^0+\epsilon'}^{b_{k,i}^0+\epsilon'}
v_{k,i}\omega\right)_{k\in\Z,1\leq i\leq n_k}$$
extends smoothly at $\tau=0$ to a smooth function
of all parameters. Moreover, its value at $\tau=0$ is
$$\left(\int_{a_{k,i}^0+\epsilon'}^{a_{k,i}} g_k^{-1}\omega_k
\right)_{k\in\Z,1\leq i\leq n_k}.$$
The function
$$\left(\int_{a_{k,i}^0+\epsilon'}^{b_{k,i}^0+\epsilon'}
w_{k,i}\omega\right)_{k\in\Z,1\leq i\leq n_k}$$
extends smoothly at $\tau=0$ to a smooth function
of all parameters. Moreover, its value at $\tau=0$ is
$$\left(\int_{b_{k,i}}^{b_{k,i}^0+\epsilon'}
g_{k+1}^{-1}\omega_{k+1}
\right)_{k\in\Z,1\leq i\leq n_k}.$$
\end{lemma}
\begin{proof}
Let us prove the first statement.
Using lemma \ref{lemma-appendix1} we have
\begin{eqnarray*}
\int_{a_{k,i}^0+\epsilon'}^{b_{k,i}^0+\epsilon'}
v_{k,i}\omega
&=& \int_{v=\varphi_{k,i}}^{\frac{t^2}{\psi_{k,i}}}
\left(\sum_{n\geq -1} c_{k,i,n}^+ r^{n+1} v^{n+1} dv
+\sum_{n\geq 2}(rt^2)^{n-1} c_{k,i,n}^- \frac{dv}{v^{n-1}}
\right)\\
&=& \frac{1}{r}\sum_{n\geq -1}\frac{c_{k,i,n}^+}{n+2}
\left(\left(\frac{rt^2}{\psi_{k,i}}\right)^{n+2}
-\left(r\varphi_{k,i}\right)^{n+2}\right)\\
&+& rt^2 c_{k,i,2}^-
(\log t^2 -\log(\varphi_{k,i}\psi_{k,i}))\\
&+& rt^2\sum_{n\geq 3} \frac{c_{k,i,n}^-}{n-2}\left(
\left(\frac{rt^2}{\varphi_{k,i}}\right)^{n-2} -\left(
r\psi_{k,i}\right)^{n-2}\right)
\end{eqnarray*}
The terms on the second and fourth line extend smoothly at
$t=0$ by the same argument as in lemma \ref{lemma-appendix2}.
The term on the third line extends continously
at $t=0$, and only as a smooth function of $\tau$ because
of the term $t^2\log t^2$.
The value of the integral when $t=0$ is
$$\frac{-1}{r}\sum_{n\geq -1}\frac{c_{k,i,n}^+}{n+2}
(r\varphi_{k,i})^{n+2}
=\int_{v=\varphi_{k,i}}^0
\sum_{n\geq -1} c_{k,i,n}^+ r^{n+1} v^{n+1} dv
=\int_{a_{k,i}^0+\epsilon'}^{a_{k,i}} g_k^{-1}\omega_k.$$
The proof of the second statement of lemma
\ref{lemma-appendix3} is similar,
using $w_{k,i}$ instead of $v_{k,i}$ as a coordinate.
\end{proof}

\medskip

We are now ready to prove lemma \ref{lemma-period-horizontal}.
Recall that $g^{(-1)^k}$ is equal to
$t g_k$ in $\C_k$, $(tg_{k+1})^{-1}$ in $\C_{k+1}$,
and $\frac{w_{k,i}}{t}$ in $W_{k,i}$. So we have
$$t\int_{B_{k,i}} g^{(-1)^k}\omega
=t^2\int_{a_{k,i}^0+\epsilon'}^{a_{k,1}^0+\epsilon'}
g_k\omega
+\int_{a_{k,1}^0+\epsilon'}^{b_{k,1}^0+\epsilon'} w_{k,1}\omega
+\int_{b_{k,1}^0+\epsilon'}^{b_{k,i}^0+\epsilon'} g_{k+1}^{-1}
\omega
+\int_{b_{k,i}^0+\epsilon'}^{a_{k,i}^0+\epsilon'}
w_{k,i}\omega.$$
The first and third terms are integrals on fixed paths, so
they depend smoothly on parameters by the smooth dependence
of $\omega$. The second and fourth terms extend smoothly
at $\tau=0$ by the previous lemma.
Moreover, the value at $\tau=0$ is
$$\int_{b_{k,1}}^{b_{k,1}^0+\epsilon'}
g_{k+1}^{-1}\omega_{k+1}
+\int_{b_{k,1}^0+\epsilon'}^{b_{k,i}^0+\epsilon'}
g_{k+1}^{-1}\omega_{k+1}
+\int_{b_{k,i}^0+\epsilon'}^{b_{k,i}}
g_{k+1}^{-1}\omega_{k+1}
=\int_{b_{k,1}}^{b_{k,i}}
g_{k+1}^{-1}\omega_{k+1}.$$
In the same way, we write
$$t\int_{B_{k,i}} g^{(-1)^{k+1}}\omega
=\int_{a_{k,i}^0+\epsilon'}^{a_{k,1}^0+\epsilon'}
g_k^{-1}\omega
+\int_{a_{k,1}^0+\epsilon'}^{b_{k,1}^0+\epsilon'}
v_{k,1}\omega
+t^2\int_{b_{k,1}^0+\epsilon'}^{b_{k,i}^0+\epsilon'}
g_{k+1}\omega
+\int_{b_{k,i}^0+\epsilon'}^{a_{k,i}^0+\epsilon'}
v_{k,i}\omega.$$
This function extends smoothly at $\tau=0$ with value
$$
\int_{a_{k,i}^0+\epsilon'}^{a_{k,1}^0+\epsilon'}
g_k^{-1}\omega_k
+
\int_{a_{k,1}^0+\epsilon'}^{a_{k,1}}
g_k^{-1}\omega_k
+\int_{a_{k,i}}^{a_{k,i}^0+\epsilon'}
g_k^{-1}\omega_k
=\int_{a_{k,i}}^{a_{k,1}}
g_k^{-1}\omega_k.$$
This proves lemma \ref{lemma-period-horizontal}.
\cqfd


\begin{thebibliography}{9}


\bibitem{HP} L. Hauswirth, F. Pacard, {\it Higher genus 
Riemann minimal surfaces}, Invent. Math., 169 (3), 569-620 (2007).

\bibitem{MT} L. Mazet, M. Traizet, {\it A quasi-periodic minimal surface},
Comment. Math. Helv. 83, 573--601 (2008).

\bibitem{MPR3} W. Meeks III, J. Perez, A. Ros, The geometry of minimal surfaces of finite genus I; curvature estimates and quasiperiodicity, J. of Diff. Geom. 66 (2004) 1--45

\bibitem{MPR} W. Meeks, J. Pérez, A. Ros,
{\it The geometry of minimal surfaces of finite
genus II; nonexistence of one limit end examples},
 Invent. Math. 158(2) 2004, 323-341.

\bibitem{MPR2} W. Meeks, J. Pérez,
{\it Properly embedded minimal planar domains with infinite topology are riemann minimal examples}, Current Developments in Mathematics (2008),
International Press, 281-346, ISBN 978-1-57146-139-1.

\bibitem{S} M. Shiffman, {\it On surfaces of stationary area bounded by two circles,
or convex curves, in parallel planes}, Ann. Math. (63) 1956, 77-90. 

\bibitem{T} M. Traizet, {\it Adding handles to Riemann's minimal surfaces}, J. Inst. Math. Jussieu, (1) 2002, 145-174. 

\bibitem{TT} M. Traizet, {\it Opening infinitely many nodes}, Preprint, 2010, arXiv:1010.4487. 

\end{thebibliography}
\end{document}